\documentclass[12pt, a4paper]{article}
\usepackage[cp866]{inputenc}
\usepackage[english, russian]{babel}
\usepackage{amsmath, srcltx}
\usepackage{amssymb}
\usepackage{amsfonts,cmtiup}

\usepackage{mathrsfs,cmtiup}
\usepackage{amsthm, eucal, eufrak}

\usepackage{srcltx}
\usepackage{amsfonts,amssymb,mathrsfs,amsmath,cmtiup}
\usepackage{amsthm, eucal, eufrak}

\newtheorem{theorem}{Theorem}[section]
\newtheorem{lemma}[theorem]{Lemma}
\newtheorem{proposition}[theorem]{Proposition}

\theoremstyle{definition}
\newtheorem*{definition}{Definition}

\newtheorem*{Index Convention}{Index Convention  index}
\newtheorem*{notation}{Notation}

\def\keywords#1{\par\medskip
\noindent\textbf{Key words.} #1}

\begin{document}
\let\le=\leqslant
\let\ge=\geqslant
\let\leq=\leqslant
\let\geq=\geqslant
\newcommand{\e}{\varepsilon }
\newcommand{ \g}{\gamma}
\newcommand{\F}{{\Bbb F}}
\newcommand{\N}{{\Bbb N}}
\newcommand{\Z}{{\Bbb Z}}
\newcommand{\Q}{{\Bbb Q}}
\newcommand{\R}{\Rightarrow }
\newcommand{\W}{\Omega }
\newcommand{\w}{\omega }
\newcommand{\s}{\sigma }
\newcommand{\hs}{\hskip0.2ex }
\newcommand{\ep}{\makebox[1em]{}\nobreak\hfill $\square$\vskip2ex }
\newcommand{\Lr}{\Leftrightarrow }

\title{Lie algebras admitting a metacyclic Frobenius group of automorphisms}

\markright{}

\author{{N.\,Yu.~Makarenko} \\ \small Sobolev Institute of Mathematics,
Novosibirsk, 630\,090, Russia and\\ \small Universit\'{e} de Haute
Alsace, Mulhouse,68093, France\\
[-1ex] \small  natalia\_makarenko@yahoo.fr \\{E.\,I.~Khukhro}\\
\small Sobolev Institute of Mathematics, Novosibirsk, 630\,090, Russia\\[-1ex] \small khukhro@yahoo.co.uk }

\date{}
\maketitle
\begin{center} UDC 512.5
\end{center}

\begin{center}{\it to
Victor Danilovich Mazurov on the occasion of his
70th birthday}
\end{center}

\begin{abstract}
Suppose that a Lie algebra  $L$ admits a finite Frobenius group
of automorphisms $FH$ with  cyclic kernel  $F$ and complement
$H$ such that the characteristic of the ground field does not divide $|H|$.
It is proved  that if the subalgebra  $C_L(F)$ of fixed points of the kernel
has finite dimension $m$ and the subalgebra  $C_L(H)$ of fixed
points of the complement
is nilpotent of class $c$, then $L$ has a nilpotent
 subalgebra of finite
codimension bounded in terms of  $m$, $c$, $|H|$, and  $|F|$
whose nilpotency class is
bounded in terms of only $|H|$ and $c$. Examples  show that
the condition of the kernel $F$ being cyclic is essential.
\end{abstract}

\keywords{Frobenius groups, automorphism,  Lie algebras,
nilpotency class}

\section{Introduction}
  Recall that a finite {\it Frobenius group} $FH$ with kernel $F$ and complement $H$ is a semidirect product of a normal subgroup $F$ and a subgroup
$H$  in which every element of $H$ acts without non-trivial
fixed points on $F$, that is,  $C_F(h)=1$ for all $h\in
H\setminus\{1\}$. The structure of Frobenius groups is well known.  In particular,
 all abelian subgroup of $H$ are cyclic, and if $F$ is a cyclic
group, then $H$ is also cyclic.

Mazurov's problem 17.72 in ``Kourovka Notebook''
\cite{kour} gave rise to a number of recent papers, where
 groups $G$ are considered admitting a Frobenius group
of automorphisms $FH$ with kernel $F$ and complement $H$ such that $F$
acts without fixed points, $C_G(F)=1$. The goal of these papers
\cite{khu08, mak-shu10,khu10al, khu-ma-shu, khu-ma-shu-DAN,
shu-a4, shu-law, khu12ja,khu12al} are  restrictions on
the order, rank, nilpotent length, nilpotency class, and
exponent of the group $G$ in terms of the corresponding properties and parameters
of the centralizer $C_G(H)$ and order $|H|$. For estimating the
nilpotency class of the group $G$ in the case of nilpotent centralizer
of the complement $C_G(H)$,  Lie ring methods are used. The corresponding
theorems on Lie rings and algebras $L$ with a Frobenius group
of automorphisms $FH$ such that $C_L(F)=0$ are also important in their own right.

A natural and important generalization of this situation is
consideration of groups and Lie rings with a Frobenius group
of automorphisms $FH$ such that its kernel $F$ has bounded cardinality or
dimension of the set of fixed points. Then the goal is
obtaining similar restrictions on a subgroup or a subalgebra
of bounded index or codimension. In the present paper
we consider the case of Lie algebras, for which strong
bounds  are obtained for the nilpotency 
 class of a subalgebra of bounded
codimension.

Suppose that a Lie algebra $L$ of arbitrary, not necessarily
finite, dimension admits a finite
Frobenius group of automorphisms $FH$ with cyclic kernel $F$  and
complement $H$ such that the subalgebra  $C_L(H)$ of fixed points
of the complement is nilpotent of class $c$. If $C_L(F)=0$, that is,
the Frobenius kernel $F$ acts {\it regularly} (without non-trivial
fixed points) on $L$,  then by the
Makarenko--Khukhro--Shumyatsky theorem~\cite{khu-ma-shu,
 khu-ma-shu-DAN} the Lie algebra $L$ is nilpotent
of class  bounded by some function depending only on
$|H|$ and $c$. In this paper we generalize the
Makarenko--Khukhro--Shumyatsky theorem to the case where the Frobenius kernel $F$
acts ``almost regularly'' on $L$. We prove that if
the dimension of $C_L(F)$ is finite and the characteristic $L$ does not divide $|H|$,
then $L$ is almost nilpotent with estimates for the codimension of a nilpotent
subalgebra and for its nilpotency class.

\begin{theorem} \label{t-lie-algebra} Let $FH$ be a Frobenius
group with cyclic kernel $F$ of order $n$ and complement $H$
of order $q$. Suppose that $FH$ acts by automorphisms on
a Lie algebra $L$ of characteristic that does not divide $q$ in such a manner that
the fixed point subalgebra  $C_L(F)$  of the kernel has finite
dimension $m$ and the fixed point subalgebra $C_L(H)$
of the complement is nilpotent of class
$c$. Then $L$ has a nilpotent subalgebra
of finite codimension bounded by some function depending
only on 
$m$, $n$, $q$, and $c$, whose nilpotency class is
bounded by some function depending only on $q$ and $c$.
\end{theorem}

There are examples showing that the result is not true  if
the kernel $F$ is not cyclic (see examples in \cite{khu-ma-shu}).
The functions of $m$, $n$, $q$, $c$ and of  $q$ and $c$ in
Theorem~\ref{t-lie-algebra} can be estimated from above explicitly,
although we do not write out these estimates here.

The proof of Theorem~\ref{t-lie-algebra} uses the method
of generalized, or graded, centralizers,
which was originally created in \cite{kh2} for almost regular
automorphisms of prime order, see  also~\cite{khmk1,khmk3,
khmk4,khmk5} and Ch.~4 in \cite{kh4}. This method consists in
  the following. In the proof of Theorem~\ref{t-lie-algebra} we can assume that
    the ground field contains a primitive $n$th root  of unity~$\omega$. Let $F=\langle \varphi\rangle$.
    Then $L$ decomposes into the direct sum
of eigenspaces
 $L_j=\{ a\in L\mid a^{\varphi}=\omega ^ja\}$, which are also
  components of a $(\Z /n\Z)$-grading: $[L_s,\, L_t]\subseteq
L_{s+t},$ where $s+t$ is calculate modulo $n$.  In each of the $L_i$,
$i\ne 0$,   certain subspaces
 $L_i(k)$ of bounded codimension  --- ``graded centralizers'' ---
 of increasing levels $k$ are successively constructed, and simultaneously
certain elements (representatives) $x_i(k)$ are fixed, all this up to a certain
$(c,q)$-bounded level. Elements of $L_j(k)$ have
a centralizer property with respect to the  fixed elements
of lower levels: if a commutator (of bounded weight) that involves exactly one element
$y_j(k)\in L_j(k)$ of level $k$ and some fixed elements $x_i(s)\in L_i(s)$ of
lower levels $s<k$ belongs to $L_0$, then this commutator is equal to~$0$.
The sought-for subalgebra is the subalgebra $Z$ generated by all
the $L_i(T)$, $i\ne 0$, of highest level $T$. The proof of the fact that the
subalgebra $Z$ is nilpotent of bounded class is based on
Proposition~\ref{kh-ma-shu-transformation}, which is
a combinatorial consequence of the Makarenko--Khukhro--Shumyatsky
theorem~\cite{khu-ma-shu, khu-ma-shu-DAN} and reduces the question of nilpotency to consideration of
commutators of a special form. Various collecting processes applied here and other arguments are based precisely on the aforementioned centralizer property.

Results on Lie algebras  (rings)  with Frobenius groups
of automorphisms are applicable to various classes of groups. In particular, it follows from the
Makarenko--Khukhro--Shumyatsky theorem~\cite{khu-ma-shu} that if a finite group (or a locally
nilpotent group, or a Lie group) $G$ admits a Frobenius
group of automorphisms $FH$ with cyclic kernel $F$ of order $n$ and
complement $H$ of order $q$ such that $C_G(F)=1$ and $C_G(H)$
is nilpotent of class $c$, then $G$ is nilpotent of $(c,q)$-bounded
class. Theorem~\ref{t-lie-algebra} is also  applicable to locally
nilpotent torsion-free groups with a metacyclic Frobenius
group of automorphisms (Theorem~\ref{t-groups1}).

We briefly describe the plan of the paper. After recalling definitions and introducing notation in \S\,2
we firstly prove in \S\,3
combinatorial consequences of the Makarenko--Khukhro--Shumyatsky theorem
(Theorem~\ref{combinatorial} and
Proposition~\ref{kh-ma-shu-transformation}), which are key in the proof of the theorem.  Then
in \S\,4 and \S\,5
 generalized centralizers and fixed elements are constructed and their basic properties are
proved. This is based on the original construction in \cite{kh2}, which, however, had to be considerably modified in accordance with the hypotheses of the problem. In \S\,6 the sought-for subalgebra is constructed  and the
nilpotency of this subalgebra is proved.
In~\S\,7  Theorem~\ref{t-groups1} on locally
nilpotent torsion-free groups is proved.

\section{Preliminaries}

We recall some definitions and notions.  For brevity we say that a certain quantity is {\it $(c,q)$-bounded}
(or, say,  {\it $(m,n,q,c)$-bounded}) if it is bounded above
by some function  depending only  $c$ and $q$ (respectively,
only on  $m$, $n$, $q$, and $c$).

Products  in a Lie algebra are called ``commutators''.
We denote by $\langle S\rangle $ the Lie subalgebra
generated by a subset~$S$.

Terms of the lower central series of a Lie algebra $L$ 
are defined by induction: $\gamma_1(L)=L$; $\gamma_{i+1}(L)=[\gamma_i(L),L].$
By definition a Lie algebra $L$ is nilpotent of class~$h$ if
$\gamma_{h+1}(L)=0$.

A simple commutator $[a_1,a_2,\dots ,a_s]$ of weight (length)
 $s$ is by definition the commutator $[\dots [[a_1,a_2],a_3],\dots ,a_s]$.
 By the Jacobi identity $[a,[b,c]]=[a,b,c]-[a,c,b]$ any (complex, repeated)
commutator in some elements in any Lie algebra
can be
expressed as a linear combination of simple commutators of
the same weight in the same elements. Using also the anticommutativity
$[a,b]=-[b,a]$, one can make sure that  in this linear combination
all simple commutators begin with some pre-assigned
element occurring in the original commutator. In particular, if
$ L=\langle  S\rangle $, then the space $L$ is generated by simple
commutators in elements of~$S$.

Let  $A$ be an additively written abelian group. A Lie algebra $L$
is \textit{$A$-graded} if
$$L=\bigoplus_{a\in A}L_a\qquad \text{ and }\qquad[L_a,L_b]\subseteq L_{a+b},\quad a,b\in A,$$
where  $L_a$ are subspaces of $L$. Elements of the subspaces $L_a$
are called \textit{homogeneous}, and commutators in homogeneous
elements  \textit{homogeneous commutators}. A subspace
 $H$ of the space $L$ is said to be \textit{homogeneous}
if $H=\bigoplus_a (H\cap L_a)$; then we set $H_a=H\cap L_a$.
Obviously,  any subalgebra or an ideal generated by homogeneous
subspaces is
 homogeneous. A homogeneous subalgebra  and the
quotient algebra by a homogeneous ideal can be regarded as
$A$-graded algebras with induced grading.

Suppose that a Frobenius group $FH$  with cyclic kernel
$F=\langle \varphi\rangle$ of order $n$ and complement $H$ of order
$q$ acts on a Lie algebra $L$ in such a way that the subalgebra
of fixed points $C_L(F)$ has finite dimension   $\dim
C_L(F)=m$, and the subalgebra of fixed points $C_L(H )$ is nilpotent
of class $c$.

Let $\omega$ be a primitive $n$th root of unity. We extend the
ground field by $\omega$ and denote by $\widetilde L$ the algebra
over the extended field.  The group $FH$ naturally acts on
$\widetilde L$, and the subalgebra of fixed points $C_{\widetilde
L}(H)$ is nilpotent of class~$c$, while the subalgebra of fixed
points $C_{\widetilde L}(F)$ has dimension $m$.

\begin{definition} We define $\varphi\hs$-{\it homogeneous components\/} $L_k$ for
$k=0,\,1,\,\ldots ,n-1$ as the eigensubspaces 
$$L_k=\left\{ a\in L\mid a^{\varphi}=\w ^{k}a\right\} .$$
\end{definition}
It is known that if the characteristic of the field does not divide $n$, then
$$
L= L_0 \oplus L_1\oplus \dots \oplus L_{n-1}
$$
(see, for example, Ch.~10 in the book~\cite{hpbl}). This decomposition
is a $({\Bbb Z}/n{\Bbb Z})$-grading due to the obvious
inclusions
$$[L_s,\, L_t]\subseteq
L_{s+t\,({\rm mod}\,n)},$$ where $s+t$ is calculated  modulo $n$.

{\bf Index Convention.} {\it Henceforth a small letter
with  index $i$ will denote an element of the $\varphi\hs$-homogeneous
component $L_i$, here the index will only indicate the
$\varphi\hs$-homogeneous component to which this
element belongs: $x_i\in L_i$. To lighten the notation we will not
use numbering indices for elements in  $L_j$, so that
different elements can be denoted by the same symbol when
it only matters to which $\varphi\hs$-homogeneous
component these elements  belong. For example, $x_1$ and $x_1$ can be
different elements of  $L_1$, so that $[x_1,\, x_1]$ can be a  nonzero element of
$L_2$. These indices will be usually considered modulo~$n$; for example, $a_{-i}\in
L_{-i}=L_{n-i}$.}

Note that in the framework of the Index Convention
a $\varphi\hs$-homogeneous commutator
 belongs to the $\varphi\hs$-homogeneous component $L_s$, where
 $s$ is the  sum modulo $n$ of the indices of all the elements occurring in this
commutator.

\section{Combinatorial theorem}

In this section we prove a certain combinatorial fact that
follows from the following
Makarenko--Khukhro--Shumyatsky theorem~\cite{khu-ma-shu}.

\begin{theorem} [Makarenko--Khukhro--Shumyatsky \cite{khu-ma-shu}]\label{kh-ma-shu10-1}
Let $FH$ be a Frobenius group with cyclic kernel  $F$
of order $n$ and complement $H$ of order $q$. Suppose that  $FH$
acts by automorphisms  on a Lie algebra 
$L$ in such a way that
$C_L(F)=0$ and the subalgebra of fixed points $C_L(H)$ is nilpotent
of class $c$. Then for some $(q,c)$-bounded number
$f=f(q,c)$ the algebra $L$ is nilpotent of class at most $f$. 
\end{theorem}

We consider a Frobenius group $FH$  with cyclic kernel
$F=\langle \varphi\rangle$ of order $n$ and complement $H$ of order
$q$ that acts on a Lie algebra $L$ in such a way that
the subalgebra of fixed points $C_L(F)$ has finite dimension
$m$ and the subalgebra of fixed points $C_L(H )$ is nilpotent of class
$c$. Since the kernel $F$ of the Frobenius group $FH$ is a cyclic
subgroup, the subgroup $H$ is also cyclic. Let $H= \langle h
\rangle$ and $\varphi^{h^{-1}} = \varphi^{r}$ for some $1\leq
r \leq n-1$. Then  $r$ is a primitive $q$th root of unity in the ring ${\Bbb Z}/n {\Bbb Z}$ and. moreover,
the image of the element
$r$ in ${\Bbb Z}/d {\Bbb Z}$ is  a primitive  $q$th root of unity for every divisor $d$ of the
number $n$, since  $h$
acts without non-trivial  fixed points on every subgroup
of the group $F$.

The group $H$ permutes the homogeneous components $L_i$ as follows:
${L_i}^h = L_{ri}$ for all $i\in \Bbb Z/n\Bbb Z$.
Indeed, if $x_i\in L_i$, then
$(x_i^{h})^{\varphi} =
x_i^{h\varphi h^{-1}h} = (x_i^{\varphi^{r}})^h =\omega^{ir}x_i^h$.

In what follows, for a given $u_k\in L_k$ we denote the element
$u_k^{h^i}$ by $u_{r^ik}$ in the framework the Index Convention,
since  ${L_k}^{h^i} = L_{r^ik}$. Since the sum over 
any
$H$-orbit belongs to the  centralizer $C_L(H)$, we have 
$u_k+u_{rk}+\cdots+u_{r^{q-1}k}\in C_L(H)$.

\begin{theorem}\label{combinatorial} Let  $FH$ be a Frobenius group with cyclic kernel
$F=\langle \varphi\rangle$ of order $n$ and complement  $H=\langle h
\rangle$ of order $q$ and let $\varphi^{h^{-1}} = \varphi^{r}$ for
some positive integer $1\leq r \leq n-1$. Let  $f(q,c)$ be the
function in Theorem~\ref{kh-ma-shu10-1}, let $\Bbb F$ be a field  containing
a primitive $n$th root of unity the  characteristic of which
does not divide $q$ and $n$, and let $L$ be a Lie algebra over $\Bbb F$.
Suppose that $FH$ acts by automorphisms on $L$ in such a way that the subalgebra of
fixed points $C_L(H)$ is nilpotent
of class $c$ and $L=\bigoplus_{i=0}^{n-1} L_i$, where $L_i=\{x\in
L\mid x^{\varphi}=\omega^ix\}$ are  $\varphi$-homogeneous components
(eigensubspaces for eigenvalues $\omega^i$).
Then  any
 $\varphi$-homogeneous
 commutator  $[x_{i_1},
x_{i_2},\ldots, x_{i_{T}}]$ with non-zero indices of weight
$T=f(q,c)+1$ can be represented as a linear combination of
$\varphi$-homogeneous commutators of the same weight $T$ each of
which, for every $s=1,\ldots, T$, includes exactly the same number
of elements of the orbit
$$O(x_{i_s})=\{x_{i_s},\,\,
\,x_{i_s}^h=x_{ri_s},\,\,\,\ldots,\,\,\,
x_{i_s}^{h^{q-1}}=x_{r^{q-1}i_s}\,\}$$
as the original
commutator,
 and contains a subcommutator with
zero sum of indices modulo $n$.
\end{theorem}

\begin{proof}
The idea of the proof consists in application of
Theorem~\ref{kh-ma-shu10-1} to a free Lie algebra with operators
$FH$. Let $\Bbb F$ be a field  containing a primitive $n$th root
of unity the  characteristic of which does not divide $q$ and $n$,
and let $n,q,r,T$ be the numbers in the hypothesis of Theorem
\ref{combinatorial}. In the ring $\Bbb Z /n{\Bbb Z}$ we choose
arbitrary non-zero (not necessarily distinct) elements $i_1,
i_2,\cdots, i_T\in \Bbb Z /n{\Bbb Z}$. We consider a free Lie
algebra $K$ over the field $\Bbb F$ with $qT$ free generators in
the set
$$Y=\{\underbrace{y_{i_1}, y_{ri_1}, \ldots, y_{r^{q-1}i_1}}_{O(y_{i_1})},\,\,\,
\underbrace{y_{i_2}, y_{ri_2}, \ldots,
y_{r^{q-1}i_2}}_{O(y_{i_2})},\ldots,
\underbrace{y_{i_{T}},y_{ri_T}, \ldots,
y_{r^{q-1}i_T}}_{O(y_{i_T})}\},$$
where the subsets
$O(y_{i_s})=\{y_{i_s}, y_{ri_s}, \ldots, y_{r^{q-1}i_s}\}$
are called the {\it $r$-orbits} of the elements $y_{i_1},
y_{i_2},\ldots, y_{i_T}$. Here, as in the Index Convention, we do not use numbering indices, that is,
all elements  $y_{r^ki_j}$ are by definition different free generators, even if
indices coincide. (The Index Convention will come into force in a  moment.) For
every $i=0,\,1,\,\dots ,n-1$ we define
the subspace $K_i$ of the algebra $K$ generated by all
commutators in the generators $y_{j_s}$ in which the sum of 
indices of the elements occurring in them is equal  to $i$ modulo $n$. Then
$K=K_0\oplus K_1\oplus \cdots \oplus K_{n-1}$.  It is also obvious that
$ [K_i,K_j]\subseteq K_{i+j\,({\rm mod\, n)}}$; therefore this is a
$({\Bbb Z} /n{\Bbb Z})$-grading. The Lie algebra $K$ also has
the natural ${\Bbb N}$-grading
with respect to the generating set $Y$:
$$K=\bigoplus_i G_i(Y),$$
where $G_i(Y)$ is the subspace generated by all
commutators of weight $i$ in elements of the generating set $Y$.

We define an action of the Frobenius group $FH$ on $K$. We set
$k_i^{\varphi}=\omega^i k_i$ for $k_i\in K_i$ and extend this
action to $K$ by linearity.  Since $K$ is the direct sum
of homogeneous $\varphi$-components  and the characteristic of the ground field does not divide $n$, we have
$$K_i=\{k\in K \,\,\mid\,\, k^{\varphi}=\omega^i k \},$$
that is,  $K_i$
is the eigensubspace for the eigenvalue~$\omega^i$.
An action of the subgroup $H$ is defined on the generating set $Y$ as  follows:
 $H$ cyclically permutes the elements of the $r$-orbits $O(y_{i_s})$,
$s=1,\ldots, T$:
$$(y_{r^ki_s})^h=y_{r^{k+1}i_s},\,\,\,\, k=0,\ldots, q-2;\,\,\,\, (y_{r^{q-1}i_s})^h=y_{i_s}.$$
Thus, the  $r$-orbit of an element $y_{i_s}$ is also
the $H$-orbit
of this element.  Clearly, $H$ permutes the components $K_i$ according to the following rule:
${K_i}^h = K_{ri}$ for all
$i\in \Bbb Z/n\Bbb Z$.

Let  $J={}_{{\rm id}}\!\left< K_0 \right>$ be the ideal
generated by the $\varphi$-homogeneous component $K_0$. By definition
the ideal $J$ consists of  all linear combinations of  commutators
in elements of $Y$ each of which contains a subcommutator  with
zero sum of indices modulo $n$. Clearly, the ideal $J$
is generated by homogeneous elements with respect to the gradings
$K=\bigoplus_i G_i(Y)$ and $K=\bigoplus_{i=0}^{n-1} K_i$ and,
consequently, is homogeneous with respect to both gradings, that is,
$$J=\bigoplus_i J\cap G_i(Y)=\bigoplus_{i=0}^{n-1}J\cap K_i.$$
Note also that the ideal $J$ is obviously $FH$-invariant.

Let $I={}_{{\rm id}}\!\left< \gamma_{c+1}(C_K(H)) \right>^F$
be the smallest $F$-invariant  ideal containing the subalgebra
$\gamma_{c+1}(C_K(H))$ (this ideal can be called the $F$-closure of the ideal generated by this subalgebra).
 We claim that the ideal  $I$  is
homogeneous with respect to the grading $K=\bigoplus_{i=0}^{n-1} K_i$.
Since $q$ is not divisible by the characteristic of the ground field $\Bbb F$,
we have the equality
$C_K(H)=\{a+a^h+\cdots+a^{h^{q-1}}\mid\,\,a\in K\}$. It is easy to see that the ideal $I$ consists of  linear combinations
of all possible elements of the form
\begin{equation}\label{form-I}
\underbrace{\Big[(u_a+u_a^{h}+\dots +u_a^{h^{q-1}}),\, (v_b+\dots
+v_b^{h^{q-1}}),\dots , (w_d+\dots +w_d^{h^{q-1}})}_{c+1},
\,y_{j_1}, \,y_{j_2},\ldots\Big]^{\varphi^i},
\end{equation}
where $u_a,v_b,\ldots,w_d$ are $\varphi$-homogeneous commutators
(possibly, of different weights) in elements of $Y$
and $y_{j_1},y_{j_2},\ldots \in Y$.

We continue using the fact that $H$ permutes the
components $K_i$ by the rule ${K_i}^h = K_{ri}$ for all $i\in \Bbb
Z/n\Bbb Z$, and denote $a_k^{h^i}$ by $a_{r^ik}$ (under the
Index Convention). It is important that then the image of a commutator in
elements of the generating set $Y$ under the action of the automorphism
$h$ is again a commutator in elements of  $Y$.
Rewriting~\eqref{form-I} in the new notation we  obtain that
the ideal $I$ consists of linear combinations
of all possible elements of the form
\begin{equation}\label{form-I-2} \underbrace{\Big[(u_a+\dots +u_{r^{q-1}a}),\, (v_b+\dots
+v_{r^{q-1}b}),\dots , (w_d+\dots +w_{r^{q-1}d})}_{c+1},
\,y_{j_1}, \,y_{j_2},\ldots\Big]^{\varphi^i},
\end{equation}
 where $u_a,v_b\ldots,w_d$ are homogeneous commutators (possibly, of different
 weights) in elements of the set  $Y$ and
$y_{j_1},y_{j_2},\ldots \in Y$.

We denote the element~\eqref{form-I-2} by $z$ and represent it as a
sum of $\varphi$-homogeneous elements of the form
$k_0+k_1+\cdots+k_{n-1}$, where $k_i\in K_i$. For every $i=0,\ldots, n-1$ we set
$z_i= \sum_{s=0}^{n-1} \omega^{-is}z^{\varphi^s}$. It is easy to verify that $z_i$ belongs to the
eigensubspace for
the eigenvalue $\omega^i$, that is,  in $K_i$. Furthermore,
$nz=\sum_{j=0}^{n-1}z_i$. Since the characteristic of the field does not divide
$n$, the element $n$ is invertible in the field $\Bbb F$, that is,
$z=1/n\sum_{j=0}^{n-1}z_i$. By comparing the two representations of $z$ we
obtain that $k_i=(1/n) z_i=(1/n)\sum_{s=0}^{n-1}
\omega^{-is}z^{\varphi^s}$. But the element $(1/n)\sum_{s=0}^{n-1}
\omega^{-is}z^{\varphi^s}$, being a linear combination of the elements
$z, z^{\varphi}, \ldots, z^{\varphi^{n-1}}$ in  $I$, also belongs to
$I$. Consequently, $k_i\in I$, that is,  the ideal $I$ is homogeneous
with respect to the grading $K=\bigoplus_{i=0}^{n-1} K_i$.

Note that  $I$ is also homogeneous with respect to the grading
$K=\bigoplus_i G_i(Y)$ and is $FH$-invariant.

We consider the quotient Lie algebra $M=K/(J+I)$. Since the ideals $J$ and $I$
are homogeneous with respect to the gradings $K=\bigoplus_i G_i(Y)$ and
$K=\bigoplus_{i=0}^{n-1} K_i$, the quotient algebra $M$ has
the corresponding induced gradings. The group $FH$
acts on $M$ in such a way that $C_M(F)=0$ and
$\gamma_{c+1}(C_M(H)) =0$. By Theorem~\ref{kh-ma-shu10-1}
the quotient algebra $K/(J+I)$ is nilpotent of $(q,c)$-bounded class
$f=f(q,c)$.  Consequently,
$$
[y_{i_1}, y_{i_2},\ldots, y_{i_{T}} ]\in J+I={}_{{\rm id}}\!\left<
K_0 \right>+{}_{{\rm id}}\!\left<\gamma_{c+1}(C_K(H)) \right>^F.
$$
This means that the commutator $[y_{i_1}, y_{i_2},\ldots, y_{i_{T}}]$ can be represented modulo the ideal $I$  as a linear combination
of commutators  of weight $T$ 
 in elements of $Y$ belonging to the
$\varphi$-homogeneous component $K_{i_1+i_2+\cdots+i_T}$ that
contain a subcommutator with zero sum of indices modulo $n$.
It is claimed that for every $s=1,\ldots, T$ any such
commutator includes exactly  one element of the orbit
$$O(y_{i_s})=\{y_{i_s}, y_{i_s}^h, \ldots, y_{i_s}^{h^{q-1}}\}.$$
 For every $s=1,\ldots,T$ we consider the homomorphism $\theta_s$
extending the mapping 
$$
O(y_{i_s})
\rightarrow 0; \qquad  y_{i_k}\rightarrow y_{i_k} \quad
\text{if}\quad k\neq s.
$$
Clearly, the kernel $\mathrm{Ker}\,\theta_s$ is equal to the ideal
generated by the orbit $O(y_{i_s})$. Furthermore, clearly,
the ideal $I$ is invariant under $\theta_s$ (as
any homogeneous ideal). We apply the homomorphism $\theta_s$ to
the commutator $[y_{i_1}, y_{i_2},\ldots, y_{i_{T}}]$ and
its representation modulo $I$
as a linear combination of commutators in elements of $Y$ of weight
$T$ that contain a subcommutator with zero sum of indices modulo $n$.
We obtain that the image of
$\theta_j([y_{i_1}, y_{i_2},\ldots, y_{i_{T}}])$ is equal to $0$, as well as the image of any commutator
containing elements of the  orbit $O(y_{i_s})$. Hence
the sum of all those commutators in the representation of the element $[y_{i_1},
y_{i_2},\ldots, y_{i_{T}}]$ that do not contain elements of the  orbit
$O(y_{i_s})$ is equal to zero, and we can exclude all these
commutators from our consideration. By applying consecutively
$\theta_s$, $s=1,\ldots, T$, and excluding commutators not containing
elements of $O(y_{i_s})$, $s=1,\ldots,
t$, in the end we obtain modulo $I$
a linear combination of commutators each of which contains at least one element from
every orbit $O(y_{i_s})$,
$s=1,\ldots, T$. Since  under these transformations the weight of commutators remains the same
and is equal to $T$, no other elements can appear, and every commutator will
contain exactly one element in every orbit $O(y_{i_s})$,
$s=1,\ldots, T$.

Thus, we proved that in the free Lie algebra  $K$
generated by elements of the set $Y$ the commutator   $[y_{i_1},
y_{i_2},\ldots, y_{i_{T}}]$  can be
represented modulo the ideal $I$ as a linear combination of commutators of weight $T$ in elements
of $Y$, and for every $s=1,\ldots, T$ any of commutators of  this linear combination contains
exactly one element
in the orbit $O(y_{i_s})$ and has a subcommutator with zero sum
of indices modulo $n$.

Now suppose that $L$ is an arbitrary Lie algebra satisfying the
hypothesis of Theorem~\ref{combinatorial}. Let $x_{i_1},
x_{i_2},\ldots, x_{i_{T}}$ be arbitrary $\varphi$-homogeneous
elements with non-zero indices in the subspaces $L_{i_1},
L_{i_2},\ldots, L_{i_{T}}$, respectively. We define the following
homomorphism $\delta$ from the free Lie algebra $K$ into $L$:
$$
\delta (y_{i_s})=x_{i_s},\,\,\,\,\,
\delta(y_{r^ki_s})=x_{i_s}^{h^k}\,\,\, \text{for}\,\,\,
s=1,\ldots,T; \,\,\,\,\,k=1,\ldots,q-1.
$$
Then
$$\delta[y_{i_1}, y_{i_2},\ldots, y_{i_{T}}]=[x_{i_1}, x_{i_2},\ldots,
x_{i_{T}}];\,\,\,\,
 \delta (I)=0;\,\,\,\, \delta (J)= {}_{{\rm id}}\!\left< L_0 \right>;\,\,\, \delta (O(y_{i_s}))=O(x_{i_s}).$$
 By applying $\delta$ to the representation
 of the commutator $[y_{i_1}, y_{i_2},\ldots, y_{i_{T}}]$ constructed above, as the image we
 obtain a representation of the commutator $[x_{i_1},x_{i_2},\ldots, x_{i_{T}} ]$ as  a linear
combination of commutators in  elements of the set
$X=O(x_{i_1})\cup O(x_{i_2})\cup \cdots \cup O(x_{i_T})$.
Since $\delta(I)=0$, every commutator in this linear
combination has weight $T$,
contains exactly the same number of elements from every orbit $O(x_{i_s})$, $s=1,\ldots, T$, as
the original commutator, and has a subcommutator with zero sum
of indices modulo $n$. The theorem is proved.
\end{proof}

We define a {\it MKhSh-transformation} of a commutator $[x_{i_1},
x_{i_2},\ldots, x_{i_l}]$ its representation according to Theorem~\ref{combinatorial} as
linear combination of simple
commutators in elements of $X=O(x_{i_1})\cup
O(x_{i_2})\cup \cdots \cup O(x_{i_T})$ that contain exactly the same number of
elements from every orbit $O(x_{i_s})$, $s=1,\ldots, T$, as the original commutator  and have
 initial segment from $L_0$ of weight $\leq
T=f(q,c)+1$, that is,  commutators  of the form
$$
[c_0,y_{j_{w+1}},\ldots, y_{j_v}],
$$
where
$$
c_0=[y_{j_1},\ldots, y_{j_w}]\in L_0, \,\,w\leq T,\,\,
j_1+j_2+\cdots+j_w=0\,({\rm mod} \, n) ,\,\, y_{j_k}\in X,
$$
with subsequent re-denoting
$$
z_{i_1}=-[c_0, y_{j_{w+1}}],\,\, z_{i_s}=y_{j_{w+s}}
\,\,\,\text{for}\,\,\, s>1.
$$

The following assertion is obtained by repeated application
of the MKhSh-transformation.

\begin{proposition} \label{kh-ma-shu-transformation} Let $FH$ be a Frobenius group
with cyclic kernel $F=\langle \varphi\rangle$ of order $n$ and
complement $H=\langle h \rangle$ of order $q$, and let
$\varphi^{h^{-1}} = \varphi^{r}$ for some $1\leq r \leq n-1$.  Let
$\Bbb F$ be a field  containing a primitive $n$th root of unity
the  characteristic of which does not divide $q$ and $n$, and let
$L$ be a Lie algebra over $\Bbb F$. Suppose that $FH$ acts by
automorphisms on $L$ in such a way that the subalgebra of fixed
points $C_L(H)$ is nilpotent of class $c$ and
$L=\bigoplus_{i=0}^{n-1} L_i$, where $L_i=\{x\in L\mid
x^{\varphi}=\omega^ix\}$ are eigensubspace for eigenvalues
$\omega^i$ of the automorphism $\varphi$. Then for any positive
integers $t_1$ and $t_2$  there exists a $(t_1,t_2, q,c)$-bounded
positive integer  $V=V(t_1,t_2, q, c)$ such that any commutator in
$\varphi$-homogeneous elements $[x_{i_1}, x_{i_2},\ldots,
x_{i_{V}}]$ with non-zero indices of weight $V$ can be represented
as a linear combination of $\varphi$-homogeneous commutators in
elements of the set $X=\bigcup_{s=1}^V O(x_{i_s})$, where
$$O(x_{i_s})=\{x_{i_s},\,\,
\,x_{i_s}^h=x_{ri_s},\,\,\,\ldots\,\,\,
x_{i_s}^{h^{q-1}}=x_{r^{q-1}i_s}\,\},$$  and every such commutator
either has a subcommutator of the form
\begin{equation}\label{f1}
[u_{k_1},\ldots, u_{k_s}],
\end{equation}
where there are $t_1$ different initial segments with zero sum
of indices modulo~$n$, that is,
$$k_1+k_2+\cdots+k_{r_i}\equiv 0\; ({\rm mod}\, n),\,\,\, i=1,2,\ldots,t_1,$$
$$1<r_1<r_2<\cdots<r_{t_1}=s,$$
or has a subcommutator of the form
\begin{equation}\label{f2}
[u_{k_0}, c_1,\ldots, c_{t_2}],
\end{equation}
where $u_{k_0}\in X$, every $c_i$ belongs to $L_0$, $i=1,\ldots,t_2$,
and has the form
$$[x_{k_1},\ldots, x_{k_i}], \,\,\,\,\,x_{k_j}\in X$$ with zero sum of indices modulo $n$
$$k_1+\cdots+k_i\equiv 0\; ({\rm mod}\, n).$$
Here we can set $V(t_1,t_2,c,q)=
\sum_{i=1}^{t_1}((f(q,c)+1)^2t_2)^i+1$.
\end{proposition}
\begin{proof} The proof practically word-for-word repeats the proof
of a proposition in~\cite{kh2} (see also Proposition 4.4.2 in
\cite{kh4}), but instead of the HKK-transformation one should repeatedly apply the
MKhSh-transformation. In contrast to the HKK-transformation, which always produces commutator in the original elements $x_{i_1}, x_{i_2},\ldots, x_{i_{V}}$, in our case after the MKhSh-transformation in commutators there may appear some images of the elements $x_{i_1}, x_{i_2},\ldots,
x_{i_{V}}$ under the action of the automorphism $h$. This detail does not affect the course of the proof,
but it precisely why the conclusion of
Proposition~\ref{kh-ma-shu-transformation} involves commutators
in elements of the $h$-orbits of the original elements.
\end{proof}

\section{Representatives and generalized centralizers}

Let $FH$ be a Frobenius group with kernel  $F=\langle
\varphi\rangle$ of order $n$ and complement  $H=\langle h\rangle$
of order $q$, and let $\varphi^{h^{-1}} = \varphi^{r}$ for some
$1\leq r \leq n-1$.  Suppose that the group $FH$  acts
by automorphisms on a Lie algebra $L$, and the subalgebra  $C_L(H)$
of fixed points of the complement is nilpotent
of class $c$, while the subalgebra $L_0=C_L(\varphi)$ of fixed points  of the kernel has
finite dimension $m$.
First suppose that $L=\bigoplus_{i=0}^{n-1} L_i$, where $L_i=\{x\in L\mid
x^{\varphi}=\omega^ix\}$ are eigensubspace for
eigenvalues~$\omega^i$ of the automorphism $\varphi$ (which is actually the main case).

We begin construction of generalized centralizers by  induction on the level
--- a parameter taking integer values from $0$  $T$, where
the number $T=T(q,c)= f(q,c)+1$ is define in
Theorem~\ref{combinatorial}. A generalized centralizer $L_j(s)$
of level $s$ is a certain subspace
 of the $\varphi\hs$-homogeneous component~$L_j$. Simultaneously with construction of
generalized centralizers we fix certain elements of them ---
representatives of various levels, --- the total number of which
is $(m,n,q,c)$-bounded.

\begin{definition} The {\it  pattern\/} of a commutator in $\varphi$-homogeneous elements
(in~$L_i$) is defined as its bracket structure together with the arrangement of
indices under the Index Convention. The {\it weight\/} of a pattern is the
weight of the commutator. The commutator itself is called the value of its
pattern on given elements.
\end{definition}

\begin{definition} Let $\vec x=(x_{i_1},\dots ,x_{i_k})$ be
some ordered tuple of elements $x_{i_s}\in L_{i_s}$,
$i_s=1,\ldots, n-1$, such that $i_1+\dots + i_k\not\equiv 0\,
({\rm mod}\, n).$ We set $j=-i_1-\dots - i_k\,({\rm mod}\, n)$ and
define the mappings 
\begin{equation}\label{vartheta}
\vartheta _{\vec x}: y_j\rightarrow  [y_j, x_{i_1},
 \dots , x_{i_k}].
\end{equation}
\end{definition}
By linearity they all are homomorphisms of the  subspace
 $L_j$ into $L_0$.  Since  $\dim L_0= m$, we have $\dim (L_j/{ \rm Ker}\,
\vartheta_{\vec x})\leq m$.

\begin{notation} Let $U=U(q,c)$ denote the number $V(T, T-1, q,c )$, where
$V$ is the function in the conclusion of
Proposition~\ref{kh-ma-shu-transformation}. \end{notation}

{\bf Definition of level 0.} At level 0 we only fix
representatives of level $0$.  First, for every pattern ${\bf P}$
of a simple commutator of weight $\leq U$ with indices $i\ne 0$ and zero
sum of indices, among all values of this pattern ${\bf P}$ on
$\varphi$-homogeneous elements in $L_{i},$ $i\ne 0$   we choose
commutators $c$ that form a basis of the  subspace spanned by all
values of this pattern on $\varphi$-homogeneous elements in
$L_{i}$, $i\ne 0$. The elements of  $L_{j}$,\, $j\ne 0$, occurring in
these fixed representations of the commutators $c$ are called
{\it representatives of level~$0$}. Representatives of  level $0$
are denoted by $x_j (0)$ under the Index Convention
(Recall that the same symbol can denote different
elements). Furthermore, together with every representative $x_j(0)\in L_j$,
$j\ne 0$, we also fix all elements of the orbit $O(x_j(0))$ of this
element under the action of the automorphism $h$ 
$$O(x_j(0))=\{x_j(0), x_j(0)^h,\,\ldots,
x_j(0)^{h^{q-1}} \},
$$
and also call them {\it representatives
of level~$0$}.  Elements of these orbits are denoted by
$x_{r^{s}j}(0):=x_j(0)^{h^{s}}$ under the Index Convention
(since $L_i^h\leq L_{ri}$).

Since the total number of  pattern ${\bf P}$ under  consideration is
$(n,q,c)$-bounded,  the dimension of $L_0$
  is at most $m$, and the number of elements  in every $h$-orbit is equal to~$q$, it follows that the number of
representatives of level  $0$ is $(m,n,q,c)\hs$-bounded.

{\bf Definition of level 1.}  We define the
 {\it generalized centralizers
 $L_j(1)$ of level~$1$} by setting, for every
$j\ne 0$,
$$L_j(1)=\bigcap_{\vec x}\,
\mbox{Ker}\, \vartheta  _{\vec x},
$$
where $\vec x=\left(x_{i_1}(0),\, \ldots ,\, x_{i_k}(0)\right) $ runs over all possible ordered tuples of length
 $k$ for all $k\leq U$  consisting of representatives of
level $0$ such that $j+i_1+\cdots +i_k\equiv 0\, (\mbox{mod}\,n).$
 Since the number of representatives of level  $0$ is $(m,n,q,c)\hs$-bounded, the intersection  here is taken over  a
$(m,n,q,c)\hs$-bounded number of subspaces
of codimension $\leq m$ in $L_j$. Hence $L_j(1)$
is a subspace of $(m,n,q,c)\hs$-bounded codimension in $L_{j}$.
 For brevity we also call elements of $L_j(1)$  {\it
centralizers of level $ 1$} and fix for then the notation
 $y_j(1)$ (under the Index Convention).

By construction every element
 $y_j(1)\in L_j(1)$ has the centralizer property
with respect to  representatives of  level 0:
$$\left[ y_j(1), x_{i_1}(0),\, \ldots ,\,
x_{i_k}(0) \right]=0,
$$
as soon as $k\leq U$ and $j+ i_1+\cdots +i_k\equiv 0\, (\mbox{mod}\, n)$.

We now fix representatives of level $1$. For every pattern
${\bf P}$ of a simple commutator
  of weight $\leq U$ with nonzero indices  and zero sum of indices modulo $n$, among
all values of the   pattern ${\bold P}$ on
 homogeneous elements in $L_{i}(1)$, $i\ne 0$,   we choose commutators
that form a basis of the subspace
 spanned by all values of  this pattern on homogeneous elements in
 $L_{i}(1)$, $i\ne 0$. The elements occurring in these
 commutators are called {\it representatives of level~$1$ } and are denoted by
$ x_{j}(1)$ (under the Index Convention). Furthermore,
for every (already fixed) representative $x_j(1)$ of level
$1$ we fix  all elements of the  $h$-orbit
$$O(x_j(1))=\{x_j(1), x_j(1)^h,\,\ldots, x_j(1)^{h^{q-1}} \},$$
and also call them
{\it representatives of level~$1$}. These elements are denoted by
$x_{r^{s}j}(1):=x_j(1)^{h^{s}}$ under the Index Convention 
 (since $L_i^h\leq L_{ri}$).

 Since the number of  pattern under consideration is $(n,q,c)$-bounded, and the
dimension of the subspace $L_0$ is  equal to $m$, the total number
of representatives of level  $1$ is  $(m,n,
q,c)\hs$-bounded.

{\bf Definition of level $\pmb{t>1}$.}
 Suppose that we have already fixed
 a $(m,n,q,c)$--bounded number of representatives of levels $<t$.
We define {\it generalized centralizers of level $t$}
by setting,  for every
 $j\ne 0$,
$$
L_j(t)=\bigcap_{\vec x} {\rm Ker}\,   \vartheta_{\vec x},
$$
where $\vec x=\left( x_{i_1}(\varepsilon_1),  \dots ,
x_{i_k}(\varepsilon_k)\right) $ runs over all possible
ordered tuples of all lengths
 $k\leq U$ consisting of representatives of (possibly different) levels
$<t$ such that
$$j+ i_1+\cdots + i_k\equiv 0\; ({\rm mod}\, n).$$

For brevity we also call elements of $L_j(t)$ {\it
centralizers of level $t$} and fix for them the notation
 $y_j(t)$ (under the Index Convention).

 The number of representatives of all levels $<t$ is
 $(m,n,q,c)\hs$-bounded and $\dim L_j/{ \rm Ker}\,\vartheta_{\vec
x} \leq m$ for all $\vec x$. Hence the intersection here is taken over
a $ (m,n,q,c)\hs$-bounded number of subspaces
of codimension $\leq m$ in $L_j$, and therefore $L_j(t)$ also has
$(m,n,q,c) \hs$-bounded codimension in the subspace
$L_{j}$.

By definition a centralizer $y_j(t)$ of level
 $t$ has the following centralizer property with respect to
representatives of lower levels:
\begin{equation}\label{centalizer-property}
\left[ y_j(t), x_{i_1}(\varepsilon_1),\, \ldots ,\,
 x_{i_k}(\varepsilon_k) \right]=0,
\end{equation}
as soon as $j+ i_1+\cdots +i_k\equiv 0\, (\mbox{mod}\, n)$, \
$k\leq U$, and the elements $x_{i_s}(\varepsilon_s)$ are representatives
of any (possibly different) levels $\varepsilon_s<t$.

We now fix representatives of level $t$. For every pattern
${\bold P}$ of a simple commutator of weight $\leq U$ with nonzero
indices and zero sum of indices, among all values of the pattern
${\bold P}$ on $\varphi$-homogeneous elements in $L_{i}(t)$, $i\ne
0$, we choose commutators that form a basis of the subspace
spanned by all the values of the  pattern ${\bold P}$ on
$\varphi$-homogeneous elements in $L_{i}(t)$, $i\ne 0$. The homogeneous
elements occurring in these commutators are called {\it
representatives of  level $t$} and are denoted by $ x_{j}(t)$ (under the Index Convention).
Next, for every (already fixed) representative $x_j(t)$ of level $t$, we fix
the elements of the $h$-orbit
$$O(x_j(t))=\{x_j(t), x_j(t)^h,\,\ldots,
x_j(t)^{h^{q-1}} \},
$$
and call them also
 {\it representatives of level~$t$}.  These elements are denoted by
$x_{r^{s}j}(t):=x_j(t)^{h^{s}}$ under the Index Convention (since $L_j^{h^s}\leq L_{r^sj}$).
Since the number of  patterns under consideration is $(n,q,c)\hs$-bounded and the
dimension of the subspace $L_0$ is equal to $m$, the total number
of representatives of level~$t$ is $(m,n, q,c)\hs$-bounded.
The construction of centralizers and representatives of levels  $\leq T$
is complete.

\section{Properties of centralizer and representatives}

Recall that we fixed the notation $T=T(q,c)=f(q,c)+1$
(for the maximal level) and  $U=V(T, T-1, q,c )$, where $f$, $V$
are functions in the conclusion of Theorem~\ref{kh-ma-shu10-1} and Proposition~\ref{kh-ma-shu-transformation}, respectively.

It is clear from the construction of generalized centralizers that
\begin{equation}\label{vkljuchenie}
L_j(k+1)\leq L_j(k)
\end{equation}
for all $j\ne 0$ and all $k=1,\ldots, T$.

The following lemma follows immediately from the definitions of level~0 and
levels~$t>0$ and from the inclusions~\eqref{vkljuchenie}; we shall usually refer to this lemma as the ``freezing'' procedure.

\begin{lemma}[freezing procedure]\label{zamorazhivanie}
Every simple commutator
$$
[y_{j_1}(k_1),y_{j_2}(k_2),\dots,y_{j_w}(k_w) ]
$$
of weight $w\leq U$
 in centralizers of levels $k_1,k_2,\dots, k_w$ with zero modulo $n$ sum of indices
$$j_1+\cdots+j_w \equiv  0\; ({\rm mod}\, n)
$$
can be represented $($frozen\/$)$ as a linear
combination of commutators
$[x_{j_1}(s),x_{j_2}(s),\dots,x_{j_w}(s) ]$ of the same pattern in
representatives of  any level $s$ satisfying
$ 0\leq s\leq \min \{ k_1,k_2,\dots,k_w \}$.
\end{lemma}

\begin{definition}  We define a {\it
quasirepresentative of weight $w$ and level $k$\/} to be any
commutator of weight $w\geq 1$ which involves exactly one
representative $x_i(k)$
 of level $k$ and $w-1$ representatives $x_s(\varepsilon_s)$ of any lower levels  $\varepsilon_s<k$.
 Quasirepresentatives of level $k$ (and only they) are denoted by
$\hat{x}_{j}(k)\in L_j$ under the Index Convention; here, obviously, the index $j$ is equal modulo $n$ sum of
indices of all
the elements occurring in the quasirepresentative. Quasirepresentatives of weight
$1$ are precisely  representatives.
\end{definition}

\begin{lemma}\label{invariance} If $y_j(t)\in L_j(t)$ is a centralizer of level
$t$, then $(y_j(t))^h$ is a centralizer of level $t$. If
$\hat{x}_j(t)$ is a quasirepresentative of level $t$, then
$(\hat{x}_j(t))^h$ is a quasirepresentative of level $t$.
\end{lemma}
\begin{proof}
Since $(y_{j}(t))^h\in L_{rj}$, we can denote $(y_{j}(t))^h$
by $y_{rj}$. Let
$x_{i_1}(\varepsilon_1), \dots ,
 x_{i_k}(\varepsilon_k)$, $k\leq
U$, be arbitrarily chosen representatives of  any (possibly different) levels $\varepsilon_s<t$ such that
$rj+i_1+i_2+\cdots+i_k\equiv 0 \; ({\rm mod}\, n) $. By construction
the elements
$(x_{i_s}(\varepsilon_s))^{h^{q-1}}=x_{r^{q-1}i_s}(\varepsilon_s)$,
$s=1,\ldots,k$,  are also representatives of  the corresponding
levels $\varepsilon_s$.  By hypothesis the element
$y_j(t)=(y_{rj})^{h^{q-1}}$
is a centralizer of level $t$; therefore it has the centralizer
property~\eqref{centalizer-property} with respect to
representatives of lower levels:
$$ [y_j(t), x_{r^{q-1}i_1}(\varepsilon_1),  \dots ,
 x_{r^{q-1}i_k}(\varepsilon_k)]=0,
$$
since
 $j+r^{q-1}i_1+\cdots+ r^{q-1}i_k \equiv  0 \; ({\rm mod}\, n)$ and $k\leq U$. By applying the
automorphism  $h$ to the last equation we obtain that
$$[y_{rj}, x_{i_1}(\varepsilon_1),  \dots ,
 x_{i_k}(\varepsilon_k)]=0,
$$
that is, the element $(y_{j}(t))^h=y_{rj}$ is a centralizer
 of level $t$.

We now consider a quasirepresentative $\hat{x}_j(t)$ of weight $k$
of level $t$. By definition this element has the form 
$$\hat{x}_j(t)= [x_{i_1}(t),\,\, x_{i_2}(\varepsilon_2), \dots ,
 x_{i_k}(\varepsilon_{k})],
$$
where  $x_{i_1}$ is a representative of level $t$
and $x_{i_2}(\varepsilon_2), \dots ,
 x_{i_k}(\varepsilon_k)$ are representatives of any (possibly different) levels
$\varepsilon_s<t$ such that $i_1+i_2+\cdots+i_k\equiv j\; ({\rm
mod}\, n)$.  By construction the elements
$(x_{i_s}(\varepsilon_s))^{h}=x_{ri_s}(\varepsilon_s)$,
$s=1,\ldots,k$,  are also representatives of the same
levels~$\varepsilon_s$. Therefore,
$$(\hat{x}_j(t))^h=[(x_{i_1}(t))^h,
(x_{i_2}(\varepsilon_2))^h, \dots , (x_{i_k}(\varepsilon_{k}))^h]=[x_{ri_1}(t),
x_{ri_2}(\varepsilon_2), \dots ,  x_{ri_k}(\varepsilon_{k})],
$$
whence $(\hat{x}_j(t))^h$ is
 also a quasirepresentative of level $t$.
\end{proof}

In what follows, when using Lemma~\ref{invariance} we shall by default denote the elements
$y_{j}(t)^{h^s}$ by $y_{r^sj}(t)$, and the elements $\hat{x}_j(t)^{h^s}$
by $\hat{x}_{r^sj}(t)$.

 Lemma~\ref{invariance} also implies that representatives of level $t$,
 elements $x_j(t), x_j(t)^h,\ldots, x_j(t)^{h^{q-1}}$,   are centralizers of
 level $t$.

\begin{lemma}\label{quasirepresentatives} Any commutator involving exactly one
centralizer ${y}_{i}(t)$ of level
 $t$ and quasirepresentatives of level $< t$ is equal to
 $0$  if the sum of indices of the elements occurring in it is equal to $0$ and the sum
of weight of all these elements is t most~$U+1$.
  \end{lemma}

\begin{proof} Based on the definitions, by the Jacobi and anticommutativity identities
we can represent this
commutator as a linear combination of simple commutators of  weight
 $\leq U+1$ beginning with the centralizer
of level $t$ and involving in addition only some representatives of  levels
$<t$. Since the sum of indices of all these elements is also equal to
$0$, all these commutators are equal to
 $0$ by~\eqref{centalizer-property}.
 \end{proof}

\section{Main theorem}\label{section-main-teorem}

 In the proof of Theorem~\ref{t-lie-algebra} the main case is when $L$ is a
$\varphi$-homogeneous  $\Bbb Z/n\Bbb
Z$-graded Lie algebra, that is,  $L=L_0\oplus L_1\oplus \cdots
\oplus L_{n-1}$.

\begin{proposition}\label{main-proposition} Theorem~\ref{t-lie-algebra} holds
for $\varphi$-homogeneous $\Bbb Z/n\Bbb Z$-graded Lie algebras
$L=L_0\oplus L_1\oplus \cdots \oplus L_{n-1}$.
\end{proposition}

\begin{proof}
 Recall that
$T$ is the fixed notation for the highest level, which is a $(q,c)$-bounded number.

 In \S\,4 we constructed the generalized centralizer $L_j(T)$. We set
$$
Z=\left< L_1(T),\,L_2(T),\ldots ,L_{n-1}(T)\right>.
$$

For every $k =0,1,\dots ,n-1 $ we denote the subspace $Z\cap
L_k$ by $Z_k$. Clearly,
$$
 Z=\bigoplus\limits_{k=0}^{n-1}Z_{k},
$$
and, in particular, $Z$ is generated by the subspaces~$Z_{k}$.
Furthermore, the subalgebra $Z$ is $H$-invariant by Lemma~\ref{invariance}
 and $(Z_k)^h=Z_{rh}$, since $(L_i)^h=L_{ri}$, $i\ne 0$.

Every subspace $L_j(T)$ has $(m,n,q,c)$-bounded
codimension in $L_j$, while the dimension of $L_0$ is equal to $m$
by hypothesis. Since $L=\bigoplus_{i=0}^{n-1} L_i$ and the subalgebra $Z$
is generated by the subspaces $L_j(T)$, $j\ne 0$, it follows that $Z$ has
$(m,n,q,c)$-bounded codimension in $L$. We  claim that
the subalgebra $Z$ is in addition nilpotent of $(c,q)$-bounded class
and therefore is a required one.

Let $U=V(T, T-1, q,c )$, where $V$ is the function in the conclusion of
Proposition~\ref{kh-ma-shu-transformation}. It is sufficient to prove that every simple commutator of weight $U$ of the form
\begin{equation}\label{9}[y_{i_1}(T),\ldots, y_{i_U}(T)],
\end{equation}
where $y_{i_j}(T)\in L_{i_j}(T)$,
 is equal to zero.  Let $X$
be the union of the $h$-orbits of  the elements $y_{i_1}(T),\ldots, y_{i_U}(T)$, that is,
$$X=\bigcup_{j=1}^U O(y_{i_j}(T)),
$$
where, recall,
$$
O(y_{i_j}(T))=\{y_{i_j}(T),\,\,\,y_{i_j}(T)^h=y_{ri_j}(T),\,\,\,\ldots\,\,\,
y_{i_j}(T)^{h^{q-1}}=y_{r^{q-1}i_j}(T)\,\}.
$$
By Proposition~\ref{kh-ma-shu-transformation}, the commutator~\eqref{9}
can be represented as a linear combination of $\varphi$-homogeneous
commutators in elements belonging   to the set $X$ each of which either has a
subcommutator of the form~\eqref{f1} 
in which there are $T$ distinct initial segments in $L_0$, or has a
subcommutator of the form~\eqref{f2} in which there are $T-1$ occurrences of
elements from $L_0$.  It is sufficient to prove that
the commutators~\eqref{f1} and \eqref{f2} are equal to zero.

We firstly consider the commutator
\begin{equation}\label{main-f-2}
[u_{k_0}, c_1,\ldots, c_{T-1}], 
\end{equation}
where $u_{k_0}\in X$, every $c_i\in L_0$ with numbering indices $i=1,\ldots,T-1$ 
has the form
$$[x_{k_1},\ldots, x_{k_i}],
$$
where $x_{k_j}\in X$ and $k_1+\cdots+k_i\equiv 0 \; ({\rm mod}\, n)$.

Using Lemma~\ref{zamorazhivanie} we  ``freeze'' every element $c_k$, where
$k=1,\ldots,T-1$,  as a linear combination of commutators of the same pattern of weight $<U$ in representatives of level $k$.  Expanding then the inner brackets by the Jacobi identity
$[a,[b,c]]=[a,b,c]-[a,c,b]$ we represent
the commutator~\eqref{main-f-2} as a linear combination
of commutators of the form
\begin{equation}\label{main-f-22}[u_{k_0}(T),\,\,x_{j_1}(1),\ldots, x_{j_k}(1),\,x_{j_{k+1}}(2),\ldots,
x_{j_s}(2), \ldots, \, \, \, \,x_{j_{l+1}}(T-1),\ldots,
x_{j_{u}}(T-1)\,].
\end{equation}
We subject the commutator~\eqref{main-f-22} to a certain collecting process.
Our aim is a representation of the commutator as a
linear combination of commutators with initial segments
consisting of
representatives of different levels $1,2,\ldots,T-1 $ and
the element $u_{k_0}(T)$. For that, by the formula
$[a,b,c]=[a,c,b]+[a,[b,c]]$,   in the commutator~\eqref{main-f-22}
we begin moving the element $x_{j_{k+1}}(2)$ (the  first from the left element
of level 2) to the left, aiming at placing it right after the element
$x_{j_{1}}(1)$.   In the course of these transformations there will appear
additional summands of special form. At first step, say, we obtain a sum
$$[u_{k_0}(T),\,\,\ldots,x_{j_{k+1}}(2),
x_{j_k}(1),\,\ldots,\,]+[u_{k_0}(T),\,\,\ldots,
[x_{j_k}(1),\,x_{j_{k+1}}(2)],\ldots].
$$
In the first summand
we continue transferring the element $x_{j_{k+1}}(2)$ to the left, over all
representatives of level 1. In the second summand we replace the subcommutator
$[x_{j_k}(1),\,x_{j_{k+1}}(2)]$ by the quasirepresentative
$\hat{x}(2)=\hat{x}_{j_k+j_{k+1}}(2)$ 
and move already this quasirepresentative to the left over all representatives of  level 1.
Since we transfer a quasirepresentative of level $2$ over
representatives of level $1$, in additional summands every time there appear subcommutators that
are quasirepresentatives of level $2$, which assume the role of the element that is being transferred.
As a result we obtain a linear combination of commutators  of the form
$$[[u_{k_0}(T),x(1),\hat{x}(2)], x(1),\ldots,x(1),\,\,x(2),\ldots,
x(2),\ldots, x(T-1),\ldots,x(T-1)]
$$
with collected initial
segment $[u_{k_0}(T),x(1),\hat{x}(2)]$. (For simplicity we omitted
indices in the formula.)  Next we begin moving to the beginning
the first from the left representative of level 3 aiming at placing it in the fourth place. It is important that this element is also transferred only
over representatives of  lower level, and
subcommutators in additional summands are quasirepresentatives of level 3. Replacing these
subcommutators by  quasirepresentatives of level 3,  we keep moving them to the left and so on.
At the end of this process we obtain
a linear combination of commutator with initial segment of the form
\begin{equation}
\label{main-f-23}[y_{k_0}(T),\, \hat{x}_{k_1}(1),
\hat{x}_{k_2}(2),\ldots,\hat{x}_{k_{T-1}}(T-1) ].
\end{equation}
By Theorem~\ref{combinatorial} the commutator~\eqref{main-f-23} of weight
$T$ is equal to a linear combination of $\varphi$-homogeneous commutators
of the same weight $T$ 
in elements of the $h$-orbits of the  elements
$y_{k_0}(T),\, \hat{x}_{k_1}(1),
\hat{x}_{k_2}(2),\ldots,\hat{x}_{k_{T-1}}(T-1)$ that 
have
subcommutators with zero sum of indices modulo $n$. By
Lemma~\ref{invariance} every element $(\hat{x}_
{k_i}(i))^{h^{l}}$ is  a quasirepresentative
of the form $\hat{x}_{r^lk_i}(i)$ of level $i$ and any $(y_{k_0}(T))^{h^l}$
is a centralizer of the form $y_{r^lk_0}(T)$ of level $T$. Since
every level appears only once in \eqref{main-f-23} and
there is an initial segment with zero sum of indices, every
commutator of this linear combination is equal to 0 by
Lemma~\ref{quasirepresentatives}.

We now show that a commutator of the form
\begin{equation}\label{main-f-1}
[y_{k_1}(T),\ldots, y_{k_s}(T)],
\end{equation}
where $y_{k_j}\in X$ and there are $T$ distinct initial segments with
zero sum of indices modulo $n$:
$$
k_1+k_2+\cdots+k_{r_i}\equiv 0\,(\rm{mod}\, n),\,\,\, i=1,2,\ldots,T,
$$
$$
1<r_1<r_2<\cdots<r_{T}=s,
$$
is equal to zero. The commutator~\eqref{main-f-1} belongs to $L_0$ and is
a commutator in elements of centralizers $L_i(T)$ of level $T$; therefore by
Lemma~\ref{zamorazhivanie} it can be ``frozen'' in
level $T$, that is,  represented as a linear combination
of commutators of the same pattern of weight $\leq U$ in
representatives of level $T$:
\begin{equation}\label{main-f-11}
[x_{k_1}(T),\ldots, x_{k_s}(T)].
\end{equation}
Next, the initial segment of the commutator~\eqref{main-f-11} of length
$r_{T-1}$ also belongs to $L_0$ and is a commutator in elements of
the centralizers $L_i(T-1)$ of level $T-1$, since
$L_i(T-1)\leq L_i(T)$; therefore by Lemma~\ref{zamorazhivanie} it can be
``frozen'' in level $T-1$, that is,  represented in the form
of a linear combination of commutators of the pattern of weight $\leq
U$ in representatives of level $T-1$, and so on. As a result we obtain
a linear combination of commutator of the form
\begin{equation}\label{main-f-12}
[x(1),\ldots, x(1),x(2),\ldots, x(2),\ldots, \ldots, x(T),\ldots,
x(T)].
\end{equation}
(We omitted here indices for
simplicity.) We subject the commutator~\eqref{main-f-12} to exactly the same
transformations as the commutator~\eqref{main-f-22}. First we transfer the left-most element
of level $2$ to the left to the second place,
then the left-most element of level 3 to the third place, and so on. In
additional summands the emerging quasirepresentatives
$\hat{x}(i)$ assume the role of the element being transferred and are also
transferred to the left to the $i$th place. In the end we obtain a linear
combination of commutators of the form
\begin{equation}\label{main-f-13}[\hat{x}_{k_1}(1),
\hat{x}_{k_2}(2),\ldots,\hat{x}_{k_{T}}(T) ].
\end{equation}
By Theorem~\ref{combinatorial} the commutator~\eqref{main-f-13} of weight
$T$ is equal to a linear combination of $\varphi$-homogeneous commutators
of the same weight $T$ 
in  elements of the $h$-orbits of the elements
$\hat{x}_{k_1}(1), \hat{x}_{k_2}(2),\ldots,\hat{x}_{k_{T}}(T)$
that have subcommutators with zero sum of indices modulo $n$.
By Lemma~\ref{invariance} every element $(\hat{x}_
{k_i}(i))^{h^{l}}$ is a quasirepresentative of the form
$\hat{x}_{r^lk_i}(i)$ of level $i$.  Let $\hat{x}_ {k_s}(s)$
be a quasirepresentative of maximal level $s$ occurring in
the initial segment with zero sum of indices. By representing
the commutator as a linear combination of simple commutators beginning
with $\hat{x}_ {k_s}(s)$ and with the same set of elements occurring in it, we obtain 0 by
Lemma~\ref{quasirepresentatives}.
\end{proof}

We now complete the proof of Theorem~\ref{t-lie-algebra}. We  shall need the following lemma.

\begin{lemma}\label{p-p}
Let $p$ be a prime number and let
$\psi $ be a linear transformation of finite order~$p^k$
of a vector space~$V$ over a field of characteristic~$p$ the space of fixed points of which
has finite
dimension~$m$. Then the dimension
 of $V$ is finite and does not exceed~$mp^k$. \end{lemma}

\begin{proof}
This is a well-known fact, the proof of which is based on considering the Jordan form of the transformation $\psi$; see,
for example, \cite[1.7.4]{kh4}.
\end{proof}

First suppose that the characteristic of the field $\Bbb F$ is equal to a prime divisor $p$ of the
number $n$. Let $ \langle \psi \rangle$ be the Sylow $p\hs$-subgroup of the group $\langle \varphi\rangle$, and let
$\langle \varphi\rangle =\langle \psi\rangle \times \langle
\chi\rangle$, where the order of $\chi$ is not divisible  by $p$. Consider the
subalgebra of fixed points $A=C_L(\chi )$. It is
$\psi$-invariant and $ C_A(\psi )\subseteq C_L(\varphi )$. Therefore,
${\rm dim\,}C_A(\psi )\leq m$, and by  Lemma~\ref{p-p},
the dimension ${\rm dim\,} A={\rm dim\,}C_L(\chi )$ is bounded
by some $(m,n)$-bounded number $u(m,n)$. Furthermore, $\chi$
is a semisimple automorphism of the Lie algebra 
$L$ of order $\leq n$. Thus, $L$ admits the Frobenius group of automorphisms
$\langle\chi\rangle H$ and ${\rm dim\,}C_L(\chi )\leq u(m,n)$.
Replacing $F$ by $\langle \chi\rangle$ we can assume that
$p$ does not divide $n$.

Let $\omega$ be a primitive $n$th root of unity. We extend the
ground field by $\omega$ and denote by $\widetilde L$ the algebra
over the extended field. The group $FH$ naturally acts on
$\widetilde L$, and the subalgebra of fixed points $C_{\widetilde
L}(H)$
is nilpotent of the same 
 class~$c$, while the subalgebra
of fixed points $C_{\widetilde L}(F)$ has the same
dimension $m$. Since the characteristic of the field does not divide $n$, we have
$$\widetilde L= L_0 \oplus L_1\oplus \dots \oplus L_{n-1},
$$
where
$$L_k=\left\{ a\in \widetilde L\mid  a^{\varphi}=\omega
^{k}a\right\},$$
and this decomposition is a $({\Bbb Z}/n{\Bbb
Z})$-grading, since
$$[L_s,\, L_t]\subseteq L_{s+t\,({\rm mod}\,n)},
$$
where $s+t$ is calculated modulo $n$.

By Proposition~\ref{main-proposition} the algebra $\widetilde L$
has a nilpotent subalgebra $Z$ of
$(m,n,q,c)$-bounded codimension and of $(q,c)$-bounded
nilpotency class. Obviously, the subalgebra  $L\cap Z$
is the sought-for subalgebra of $(m,n,q,c)$-bounded codimension and of $(q,c)$-bounded
nilpotency class in $L$. The theorem is
proved.

\section{Locally nilpotent torsion-free groups}

Every locally nilpotent torsion-free group $G$ can be embedded
into a divisible group $\sqrt{G}$ consisting of of all roots of non-trivial 
element of $G$, the so-called Mal'cev completion of $G$
(see, for example, \cite[Ch.~10]{kh5}). Every automorphism of the group $G$
can be uniquely extended to an automorphism of the group
$\sqrt{G}$. Divisible torsion-free groups can be regarded as
${\Bbb Q}\hs$-groups with additional operations of extracting rational roots.
The Mal'cev correspondence given by the Baker--Hausdorff
formula and its inversions establishes an equivalence of the category of locally nilpotent $\Q\hs$-groups and the category of locally nilpotent Lie $\Q\hs$-algebras 
 (see,
for example,~\cite[Ch.~10]{kh5}.)  We can assume that the corresponding
objects in these two categories have the same underlying  set. Let $G$ and $L$ be category equivalent
a $\Q\hs$-group and a Lie $\Q\hs$-algebra, respectively,  with the same underlying set.
Then $\Q\hs$-subgroups of  $G$ (that is,
divisible subgroups) are (as subsets) $\Q\hs$-subalgebras of the algebra $L$ and vice versa;
normal $\Q\hs$-subgroups of  $G$ are precisely ideals of $L$, and so on. The nilpotency class of a subgroup 
of  $G$
coincides with its nilpotency class as a Lie subalgebra 
of $L$.

Recall that a group has finite rank $r$ if every
finitely  generated subgroup of it is generated by $r$ elements (and
$r$ the smallest number with this  property). By Mal'cev's theorem
\cite[Theorem~5]{ml} a locally nilpotent torsion-free group $G$ has finite rank if and only if it is
nilpotent and has finite sectional rank. We shall need
the following version of Mal'cev's theorem proved in \cite{khmk4}.

\begin{lemma}[{\cite{khmk4}, Lemma 9}] \label{rk} If a locally nilpotent torsion-free group
$ C$ has finite rank $r$, then the Lie $\Q\hs$-algebra $U$ that is equivalent to $\sqrt{C}$
under the Mal'cev category correspondence has finite
$r$-bounded dimension.
\end{lemma}

\begin{theorem}\label{t-groups1} Let $FH$ be a Frobenius
group with cyclic kernel $F$ of order $n$ and with complement $H$
of order $q$. If $FH$ acts by automorphisms on a locally
nilpotent torsion-free group
$G$ in such a way that the subgroup of fixed points $C_G(F)$ has
finite rank $r$ and the subgroup of fixed points $C_G(H)$
is nilpotent of class $c$, then $G$ has a nilpotent subgroup
$T$ of nilpotency class bounded by some
function depending only on $q$ and $c$ such that $T$ has
finite ``corank'' $t=t(r,n,q,c)$ in $G$  bounded above in terms of
 $r$, $n$, $q$, $c$ in the sense that there are $t$ element $g_1,\ldots ,g_t$
 such that every element of $G$ is a root of an element of the subgroup
$ \left< g_1,\ldots ,g_t,\,T\right>$.
\end{theorem}

\begin{proof}
 We denote by the same letters $F$ and $H$ the extensions of the groups of  automorphisms
 to the group $\sqrt{G}$. Let $L$ be the Lie $\Q\hs$-algebra with the same underlying set $\sqrt{G}$
constructed by the Mal'cev correspondence. The automorphisms of the Lie algebra $L$ are the automorphisms of the
group $\sqrt{G}$ acting on the same set in exactly the same way.  Since
 $$
\sqrt{C_G(H )}=C_{\sqrt{G}}(H)=C_L(H),\,\,\,\,\, \sqrt{C_G(F)}=C_{\sqrt{G}}(F)=C_L(F),
$$
the subalgebra $C_L(H)$ is nilpotent of class $c$ and the subalgebra
$C_L(F)$ has $r$-bounded dimension by Lemma~\ref{rk}. By
Theorem~\ref{t-lie-algebra} the algebra $L$ has a nilpotent subalgebra
$Z$ of  $(c,q)$-bounded nilpotency class and of
$(r,n,q,c)$-bounded codimension. The intersection $Z\cap
G$ is a sought-for subgroup of $G$.\end{proof}

\end{document}